%% file: 000_main.tex
\documentclass{amsart}
\usepackage[utf8]{inputenc}

\input{000_setup}

\title[Cramer-Castillon on a Triangle's Incircle and Excircles]{Cramer-Castillon on a Triangle's\\Incircle and Excircles}
\author{Dominique Laurain}
\address{D. Laurain, Enseeiht,
Toulouse, France, \texttt{dominique.laurain31@orange.fr}}
\author{Peter Moses}
\address{P. Moses, Moparmatic Inc., Worcestershire, England, \texttt{moparmatic@gmail.com}}
\author{Dan Reznik$^*$} 
\address{D. Reznik, Data Science Consulting Ltd., Rio de Janeiro, Brazil. \texttt{dreznik@gmail.com}}

\begin{document}

\maketitle

\input{050_abstract}

\section{Introduction}
\label{sec:intro}
\input{100_intro}

\section{Cramer-Castillon on the Incircle}
\label{sec:incircle}
\input{110_CC_incircle}

\section{Shared Brocard objects}
\label{sec:brocard}
\input{120_CC_schoute}

\section{Cramer-Castillon on the Excircles}
\label{sec:excircles}
\input{130_CC_excircles}

\section*{Acknowledgements}
\input{210_ack}

\appendix

\section{Center Correspondences}
\label{app:cc-correspond}
\input{230_cc_correspond}

\bibliographystyle{maa}
\bibliography{999_refs,999_refs_rgk}

\end{document}

%% file: 000_setup.tex
\usepackage{graphics}
\usepackage{thmtools}
\usepackage{wasysym}
\usepackage[T1]{fontenc}    

\usepackage{amsthm}
\usepackage{amsbsy,amsmath,amssymb,amscd,amsfonts}
\usepackage[pagebackref=true]{hyperref}
\usepackage[nameinlink,capitalize,noabbrev]{cleveref}

\usepackage{graphicx,float,latexsym,color}
\usepackage[font={scriptsize,it}]{caption}
\usepackage{subcaption}

\usepackage{makecell}

\usepackage[dvipsnames]{xcolor}

\newtheorem*{theorem*}{Theorem}

\newtheorem{proposition}{Proposition}

\newtheorem{corollary}{Corollary}

\theoremstyle{remark}

\theoremstyle{definition}

\hypersetup{
    pdftoolbar=true,        
    pdfmenubar=true,        
    pdffitwindow=false,     
    pdfstartview={FitH},    
    colorlinks=true,       
    linkcolor=OliveGreen,          
    citecolor=blue,        
    filecolor=black,      
    urlcolor=red           
}

\usepackage{lineno}

\usepackage{comment}
\usepackage{microtype}
\usepackage{footnote}

\newcommand{\I}{\mathcal{I}}
\newcommand{\B}{\mathcal{B}}

\newcommand{\C}{\mathcal{C}}

%% file: 050_abstract.tex
\begin{abstract}
The Cramer-Castillon problem (CCP) consists in finding one or more polygons inscribed in a circle such that their sides pass cyclically through a list of $N$ points. We study this problem where the points are the vertices of a triangle and the circle is either the incircle or one of the excircles. We find that (i) in each case there is always a pair of solutions (total of 8 new triangles and 24 vertices); (ii) each pair shares all Brocard geometry objects, (iii) barycentric coordinates are laden with the golden ratio; and (iv) simple operations on the barycentrics of a single vertex out of the 24 yield all other 23. 

\vskip .3cm
\noindent\textbf{Keywords} Golden ratio, triangle, Brocard, symmedian.
\vskip .3cm
\noindent \textbf{MSC} {51M04
\and 51N20 \and 51N35}
\end{abstract}

%% file: 100_intro.tex
The Cramer-Castillon problem (CCP) consists in finding one or more $N$-gons inscribed in a circle $\C$ such that their sides pass cyclically through a set of points $P_i$, $i=1 \cdots N$. In \cref{fig:cc} this is illustrated for the $N=3$ case. The solutions to CCP are given\footnote{In the hyperbolic plane, corresponding sides of the two solutions are polar-orthogonal with respect to the ideal circle \cite{arnold2018-cross-ratio}.} by the roots of a quadratic equation (see
\cite[Section 6.9]{ostermann2012} \cite{wanner2006}), i.e., there can be 0, 1, or 2 real solutions. Geometric conditions for solution existence, though not germane to this article, are described in the aforementioned references.

\begin{figure}
    \centering
    \includegraphics[width=\textwidth]{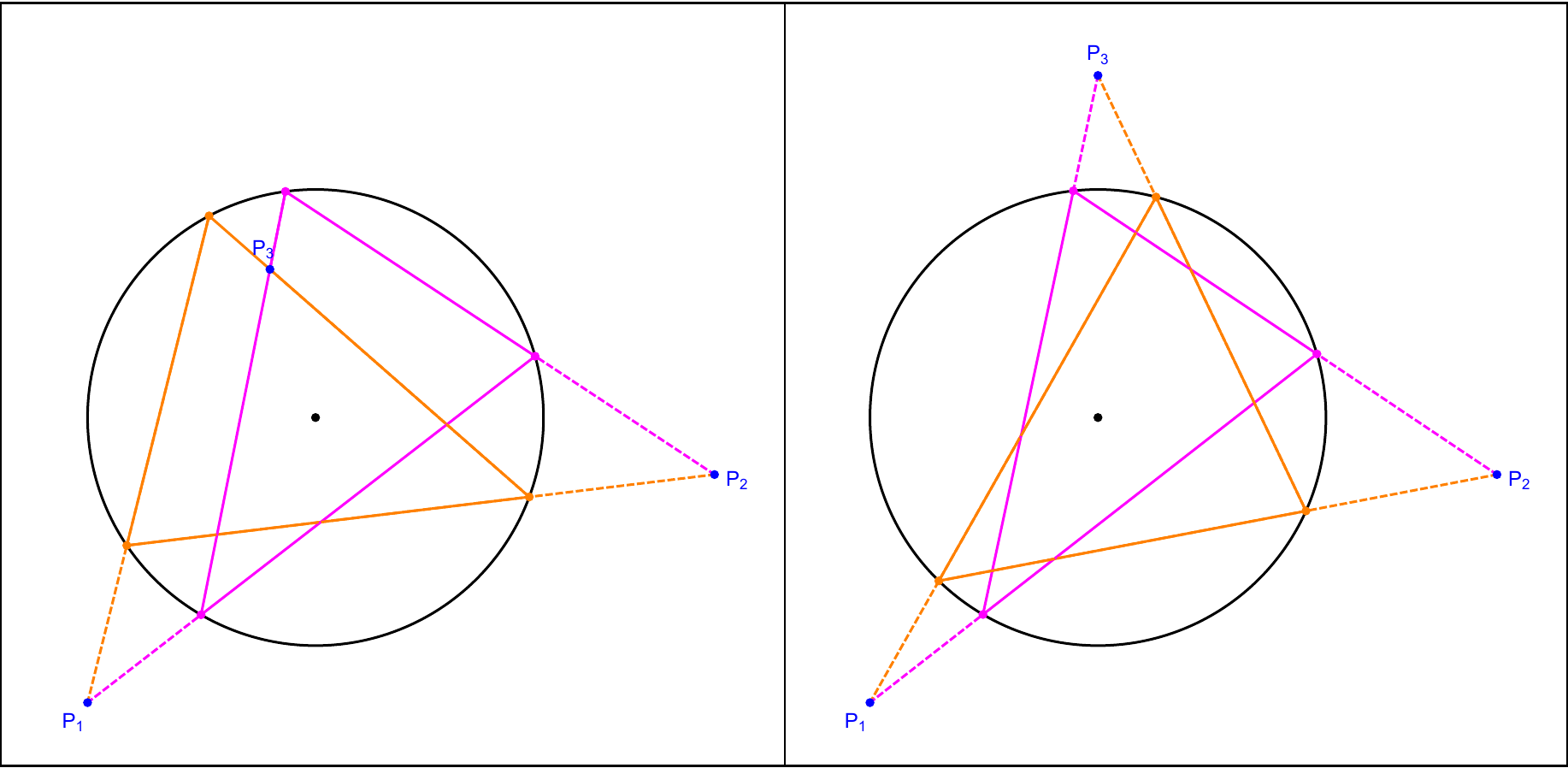}
    \caption{The Cramer-Castillon problem (CCP) in the $N=3$ case. In the left (resp. right) picture, two points are exterior and one is interior (resp. all exterior) to the target circle. In each case, two solutions are shown (magenta and orange).}
    \label{fig:cc}
\end{figure}

Referring to \cref{fig:cc-inc}, we investigate the CCP for the case where the $P_i$ are the vertices $A,B,C$ of a reference triangle and $\C$ is the incircle or one of the excircles. Our findings include:

\begin{itemize}
    \item For any triangle, the CCP on either the incircle or an excircle has exactly two solutions (total of 8 new triangles and 24 new vertices).
    \item We derive barycentric coordinates for the 4 pairs of solutions and notice they are laden with the golden ratio $\phi=(\sqrt{5}+1)/2$.
    \item Each pair shares circumcenter, symmedian point, and  all ``Brocard geometry'' objects \cite{johnson1960}, e.g., the Brocard points, Brocard circle and inellipse, Lemoine and Brocard axis, isodynamic points, etc. 
    \item The four distinct Brocard axes shared by each pair concur on the de Longchamps  point \cite{mw} of the reference triangle.
    \item Given barycentrics for a single vertex out of the 24 newly generated, all other 23 can be obtained with simple cyclic substitutions.
    \item Solving the CCP for a triangle's arbitrary inconic is equivalent to solving it (via a projectivity) for the incircle case.
\end{itemize}

\subsection*{Article organization}

The Cramer-Castillon Problem (CCP) on the incircle is covered in \cref{sec:incircle}. Its shared Brocard objects are examined in \cref{sec:brocard}. The CCP on the excircles are analyzed in \cref{sec:excircles}. In \cref{app:cc-correspond} we provide a list of correspondences between triangle centers in the incircle CCP solutions and the reference triangle.

\subsection*{Notation}

We shall use barycentric coordinates \cite{mw} and refer to triangle centers using Kimberling's notation $X_k$ \cite{etc}.

%% file: 110_CC_incircle.tex
Referring to \cref{fig:cc-inc}, consider a triangle $T=ABC$ with the sidelengths $a,b,c$ and $s=(a+b+c)/2$ its semiperimeter. Below we use $\phi$ to denote the golden ratio, $\phi  = (1 + \sqrt{5})/2$.

\begin{figure}
\centering
\includegraphics[width=\textwidth]{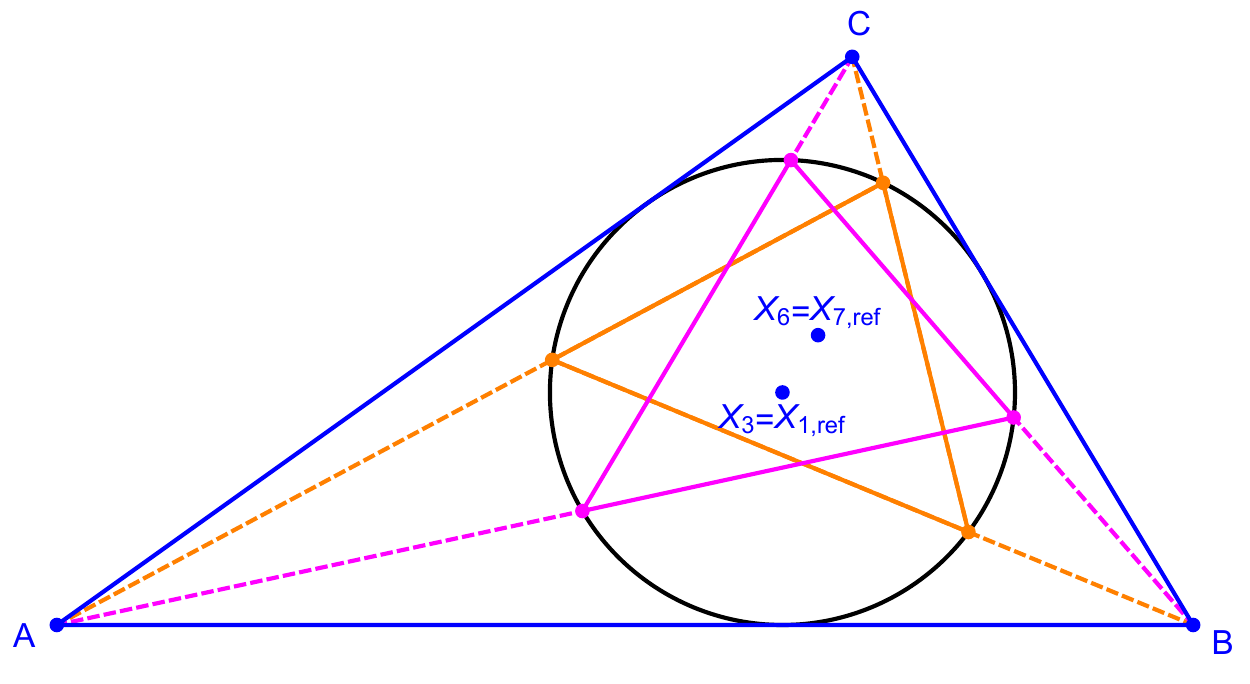}
\caption{The two solutions of CCP (orange, magenta) on the incircle $\C$ of a triangle $T=ABC$. Since both are inscribed in $\C$, they share their circumcenter $X_3$ (at the incenter $X_{1,\text{ref}}$ of the reference). Also shared is the symmedian point $X_6$, which coincides with the Gergonne $X_{7,\text{ref}}$ of the reference.}
\label{fig:cc-inc}
\end{figure}

\begin{proposition}
The CCP on a triangle $T$ and its incircle $\C$ admits exactly two solutions $T_1$ and $T_2$, whose barycentric vertex matrices with respect to $T$ are given by:
\begin{align*}
T_1&=
\begingroup 
\setlength\arraycolsep{10pt}
\begin{bmatrix}
(1 - \phi)^2  v  w & u w &  (2 - \phi)^2 u  v  \\
(2 - \phi)^2  v  w & (1 - \phi)^2  u  w &  u  v \\
 v  w & (2 - \phi)^2  u  w & (1 - \phi)^2  u  v\\
\end{bmatrix}\endgroup\\
T_2&=
\begingroup 
\setlength\arraycolsep{10pt}\begin{bmatrix}
(1 - \phi)^2  v  w & (2 - \phi)^2  u  w &  u  v \\
 v  w & (1 - \phi)^2  u  w & (2 - \phi)^2  u  v \\
(2 - \phi)^2  v  w &  u  w & (1 - \phi)^2  u  v \\
\end{bmatrix}\endgroup
\end{align*}
where $u = (s-a)$, 
$v = (s-b)$, 
$w = (s- c)$.
\label{prop:two-sols}
\end{proposition}

\begin{proof}
While barycentric entries in $T_1,T_2$ can be obtained as roots of a quadratic equation \cite[Section 6.9]{ostermann2012}, we also provide a synthetic construction for the vertices.

Referring to \cref{fig:two-sols}, define three paths ($A_1$,$A_2$,$A_3$,$A_4$), ($B_1$,$B_2$,$B_3$,$B_4$) and ($C_1$,$C_2$,$C_3$,$C_4$). Cramer-Castillon requires that $A_iA_{i+1}$ (or $B_iB_{i+1}$ or $C_iC_{i+1}$) be a circle chord.

\begin{figure}
\centering
\includegraphics[width=\linewidth]{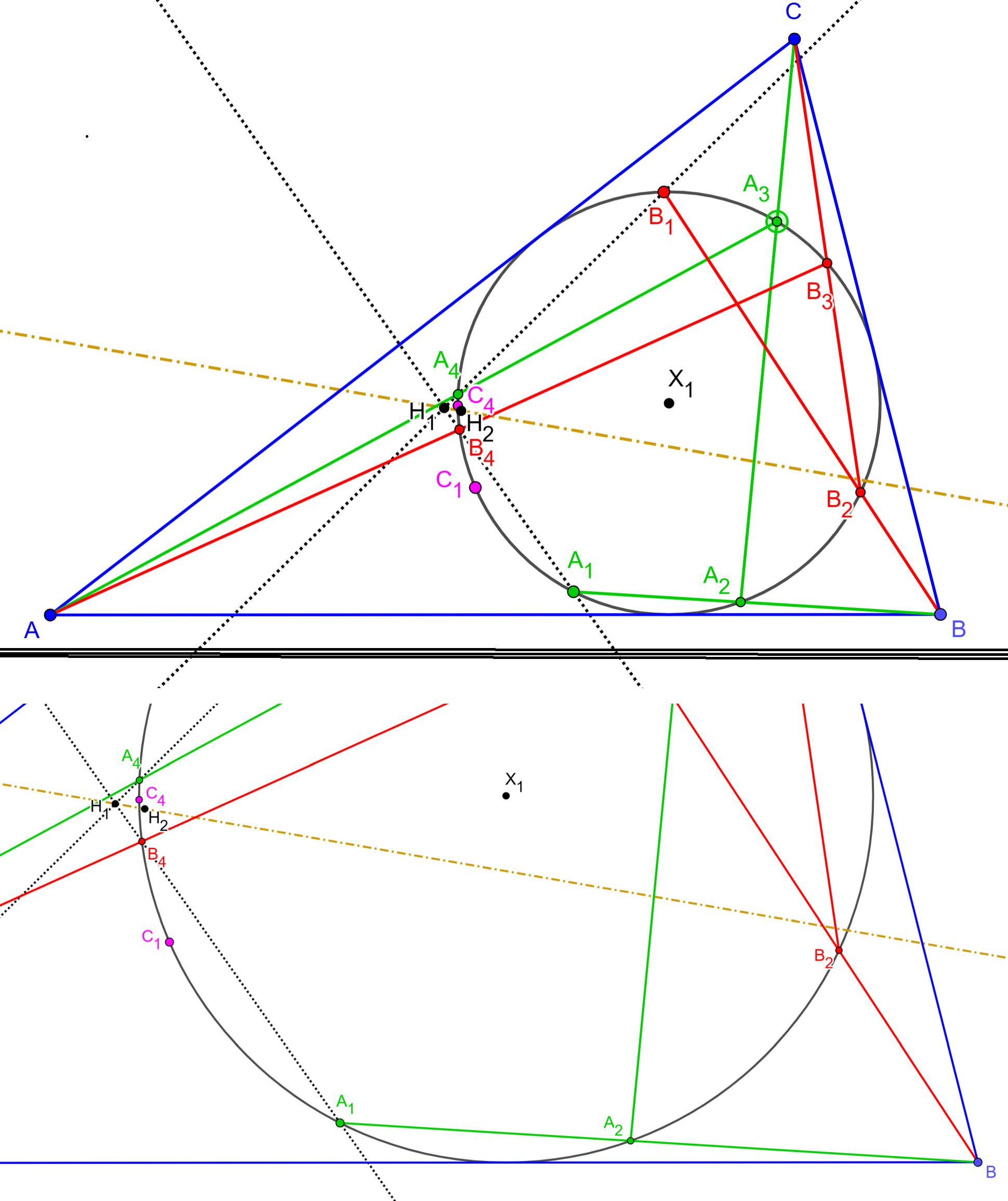}
\caption{\textbf{top}: The construction for the two points $H_1$ and $H_2$ which define the homography axis used in \cref{prop:two-sols}. The sequence $C_1, ... C_4$ is not shown, but is as $A_1 ... A_4$ and $B_1 ... B_4$. \textbf{bottom}: a slight zoom-in on the region where $H_1$ and $H_2$ are.}
\end{figure}

The points $B_1$, $C_1$ are tangent points of $\C$ with $AC$ and $AB$ while $A_1$ is the reflection of a tangent point with $BC$ with respect to $X_1$.

The three paths are non-closing: $A_4$, $B_4$ and  $C_4$ are respectively different from $A_1$, $B_1$ and $C_1$.

Intersect the two pairs of segments $(A_1B_4$,$A_4B_1)$ and $(A_1C_4$,$A_4C_1)$ to get points $H_1$ and $H_2$ on their perspectrix.

Intersecting the homography line with $\C$ gives $M_1$ and $M_4$, two vertices of the solution triangles $M_1M_2M_3$ and $M_4M_5M_6$.
\end{proof}


By definition, both solutions share their circumcenter $X_3$, located at the incenter $X_1$ of the reference. Interestingly:

\begin{proposition}
The two solutions $T_1$ and $T_2$ have a common symmedian point $X_6$ which coincides with the Gergonne point $X_7$ of $T$.
\label{prop:x6}
\end{proposition}

\begin{proof}
Using a CAS, we obtain $[1/(s-a) : 1/(s-b) : 1/(s-c)]$ as barycentrics for the symmedian point of $T_1,T_2$ with respect to $T$, which are precisely those of the Gergonne $X_7$ of $T$ \cite{etc}.
\end{proof}

%% file: 120_CC_schoute.tex
Recall (i) the Brocard circle of a triangle has segment $X_3 X_6$ as diameter, and (ii) the Brocard axis is the line that passes through said diameter \cite{mw}.

Since by definition $T_1,T_2$ share their circumcenter $X_3$, and per \cref{prop:x6} their symmedian as well:

\begin{corollary}
$T_1$ and $T_2$ share their Brocard axis and Brocard circles.
\end{corollary}

Let $\delta=|X_6-X_3|$. The Brocard angle $\omega$ to a triangle is given by \cite[Prop. 3, p. 209]{casey1888}:
\[ \tan\omega = \frac{\sqrt{3}}{3}\sqrt{1-\left(\frac{\delta}{R}\right)^2} \]

\begin{corollary}
$T_1$ and $T_2$ have the same Brocard angle $\omega$.
\end{corollary}

Referring to \cref{fig:cc-broc}, recall the two Brocard points of a triangle lie on the Brocard circle and the line joining them is perpendicular to the Brocard axis $X_3 X_6$. In  \cite{shail1996-brocard} the following formula is given for the distance between the two Brocard points $\Omega_1,\Omega_2$ which only depends on the circumradius $R$ and the Brocard angle
$\omega$:
\[|\Omega_1-\Omega_2|^2=4c^2=4R^2\sin^2\omega (1-4\sin^2\omega)\]

\begin{corollary}
$T_1$ and $T_2$ share their Brocard points $\Omega_1$ and $\Omega_2$.
\end{corollary}

\begin{proposition}
The shared Brocard points $\Omega_1$ and $\Omega_2$ of $T_1$ and $T_2$ are triangle centers of $T$ given by the following barycentric coordinates:
\[
\Omega_1 = [\alpha/u : \beta/v : \gamma/w ],\;\;\;
\Omega_2 = [\gamma/u : \alpha/v :\beta/w ]
\]
where:
\begin{align*}
\alpha &=(a - b)^2 - (a + b)c\\
\beta &=(b - c)^2 - (b + c)a\\
\gamma &=(c - a)^2 - (c + a)b
\end{align*}
and $u,v,w$ are as in \cref{prop:two-sols}.
\end{proposition}

Let $a,b$ be the semi-axes of the Brocard inellipse of a triangle (whose foci are the Brocard points). The following relation was derived in \cite[Lemma 2]{reznik2022-matryoshka}:
\[ [a,b]= R\left[\sin\omega,2\sin^2\omega\right] \]

\begin{corollary}
$T_1,T_2$ share their Brocard inellipses.
\label{cor:broc}
\end{corollary}

Recall the two isodynamic points $X_{15}$ and $X_{16}$ are the limiting points of the {\em Schoute pencil}  \cite{johnson17-schoutte}, defined by the circumcircle and the Brocard circle (and orthogonal to the Apollonius circles), whose radical axis is the Lemoine axis.

\begin{corollary}
$T_1$ and $T_2$ share their isodynamic points $X_{15}$ and $X_{16}$ and Lemoine axis.
\end{corollary}

The intersection of the Lemoine axis with the Brocard axis is $X_{187}$  \cite{etc}.

\begin{corollary}
$T_1$ and $T_2$ share their $X_{187}$.
\end{corollary}

\begin{figure}
    \centering
    \includegraphics[width=\textwidth]{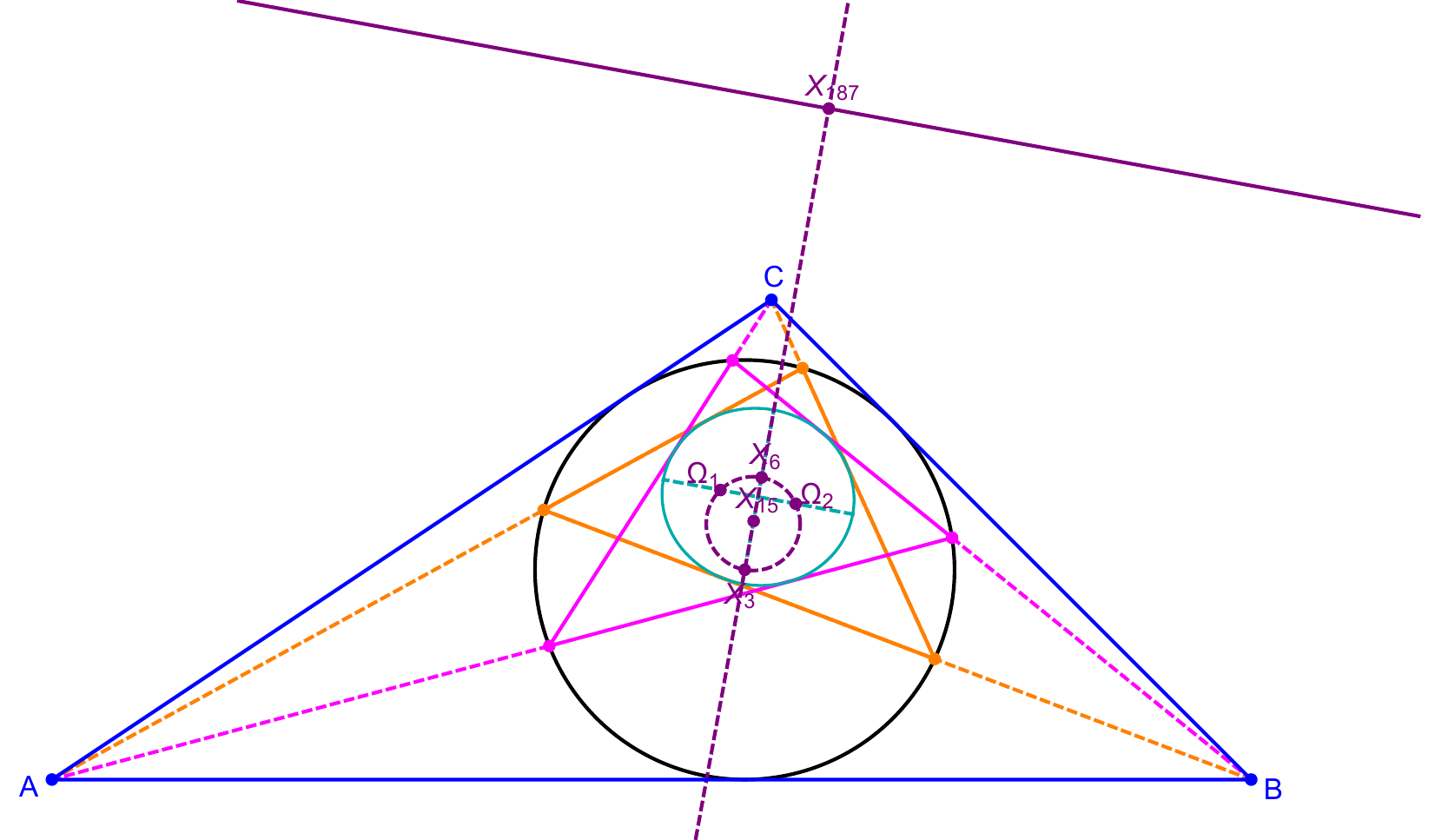}
    \caption{The two solutions (orange, magenta) of CCP on a triangle's incircle (black) share all Brocard geometry objects, to be sure: their Brocard points $\Omega_1,\Omega_2$, Brocard circle and axis (dashed purple), Brocard inellipse (light blue) whose foci are the Brocard points, the two isodynamic points $X_{15}$, and $X_{16}$ (not shown) and the Lemoine axis (solid purple).}
    \label{fig:cc-broc}
\end{figure}

\subsection*{Generalizing the CCP to any inconic} Referring to \cref{fig:cc-inconic}(top), let $\I$ be some inconic of a triangle $T=ABC$.

\begin{proposition}
The CCP of $A,B,C$ on $\I$ has two solutions which circumscribe a single inconic $\B$.
\end{proposition}

\begin{proof}
The CCP is projectively-invariant since only incidences are involved. Let $T'$ be the image of $T$ under a projectivity $\Pi$ that sends $\I$ to a circle $\C'$, see \cref{fig:cc-inconic}(bottom). Clearly, $\C'$ is the incircle or an excircle of $T'$. Per \cref{prop:two-sols,cor:broc}, the CCP on $(T',\C')$ has two solutions with a common Brocard inellipse $\B'$. Thus $\B$ is the latter's pre-image under $\Pi$, with all tangencies preserved.
\end{proof}

\begin{figure}
\centering
\includegraphics[width=\textwidth]{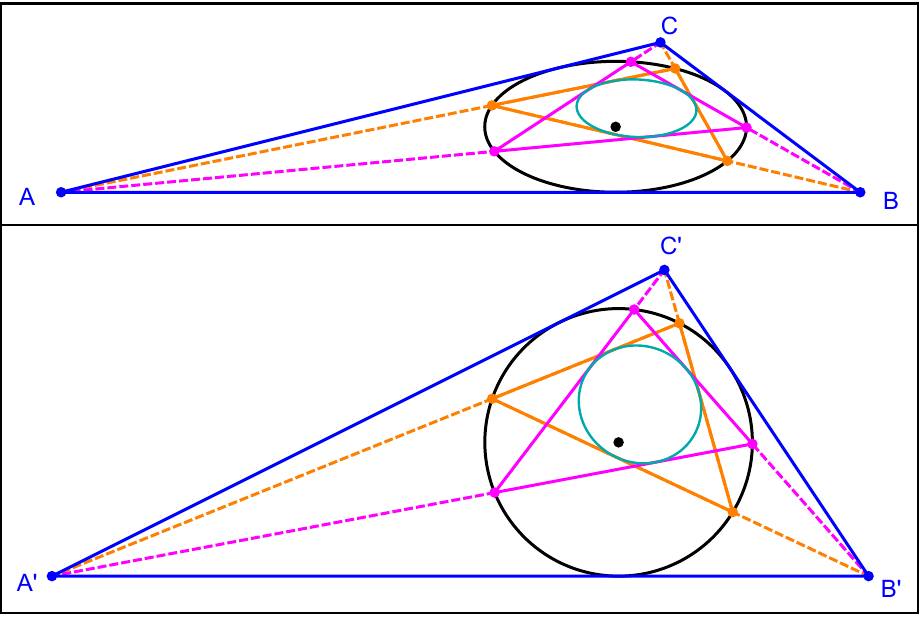}
\caption{Since the CCP is projectively-invariant, its solution on a $ABC$ (top) with respect to a generic inconic(black) can be regarded as the pre-image of a perspectivity $\Pi$ which sends said inconic to a circle (black, bottom). Clearly, this circle is an incircle of the new triangle $A'B'C'$. This implies that the two solutions in the original case will envelop a conic (light blue, top) which is the pre-image of the Brocard inellipse (light blue, bottom) under $\Pi$.}
\label{fig:cc-inconic}
\end{figure}

\begin{proposition}
Consider a general inconic $\I$ with perspector $[p,q,r]$. The two solutions $T_1$ and $T_2$ of the CCP on $\I$ have $A$-vertices given by:
\begin{align*}
T_1 \text{(A-vertex)} = & \left[(2 - \phi) p, q, (1 + \phi) r\right]\\
T_2 \text{(A-vertex)} = & \left[(2 - \phi) p, (1 + \phi) r, q\right]
\end{align*}
\end{proposition}

\begin{corollary}
If $\I$ is the Steiner inellipse (perspector $X_2$), the two A-vertices are given by:
\begin{align*}
T_1 \text{(A-vertex)} = & \left[2 - \phi, 1, 1 + \phi\right]\\
T_2 \text{(A-vertex)} = & \left[2 - \phi, 1 + \phi, 1\right]
\end{align*}

\end{corollary}

Notice that if $\I$ is the incircle, whose perspector is the Gergonne point $X_7:\left[1/(s-a) : 1/(s-b) : 1/(s-c)\right]$, we recover \cref{prop:two-sols}.

%% file: 130_CC_excircles.tex
In this section we extend the CCP to the excircles of a given triangle $T$.  Referring to \cref{fig:cc-excs}:

\begin{figure}
    \centering
    \includegraphics[width=\textwidth]{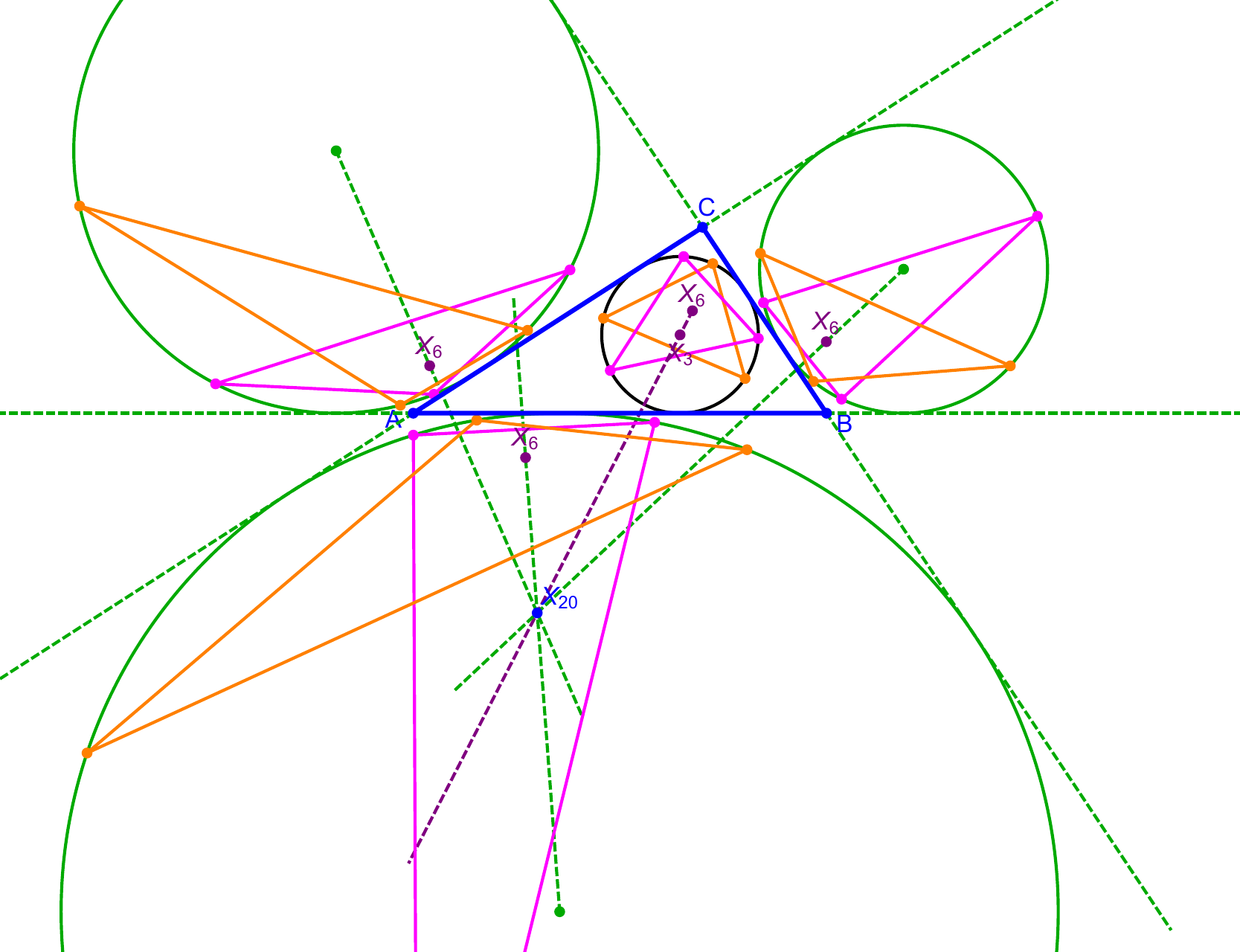}
    \caption{The CCP applied both to the incircle and the excircles (green). In each of the four circles the symmedian (and all Brocard objects) are shared. The four Brocard axes concur on the de Longchamps point $X_{20}$ of the reference triangle.}
    \label{fig:cc-excs}
\end{figure}

\begin{proposition}
The first and second solutions $T_{1,exc}$ and $T_{2,exc}$ of the CCP on the A-excircle are given by the following barycentric vertex matrices:

\[
T_{1,exc} =
\begingroup 
\setlength\arraycolsep{10pt}
\begin{bmatrix}
  -vw (\phi-1)^2 & s v &  s w (\phi-2)^2  \\
 -vw (\phi-2)^2 & s v (\phi-1)^2 & s w \\
 - v w & s v (\phi-2)^2 & s w (\phi-1)^2 \\
\end{bmatrix}\endgroup
\]

\[
T_{2,exc} =
\begingroup 
\setlength\arraycolsep{10pt}
\begin{bmatrix}
- v w (\phi-1)^2 & s v (\phi-2)^2 &  s w \\
 - v w & s v (\phi-1)^2 & s w (\phi-2)^2 \\
- v w(\phi-2)^2  & s v & s w (\phi-1)^2  \\
\end{bmatrix}\endgroup
\]
where $u,v,w$ are as in \cref{prop:two-sols}.
\end{proposition}

\begin{proposition}
The symmedian $X_6$ of the A-excircle is shared by the two solution and is the A-exversion of the reference's Gergonne $X_7$, i.e., its barycentrics are given by $\left[-v w, s v , s w \right]$, coinciding with the Gergonne point $X_7$ of the ``outer'' contact triangle inscribed in the A-excircle.
\end{proposition}

\begin{corollary}
Each pair of solutions of the CCP on an excircle shares all of their Brocard objects.
\end{corollary}

\begin{proposition}
The Brocard axis of the incircle-CCP as well as the 3 Brocard axes of the excircle-CCP solutions concur on the de Longchamps point $X_{20}$ of the reference. 
\end{proposition}

\begin{proof}
The shared Brocard axis of the incircle-CCP solutions contains, by definition, $X_3$ and $X_6$ of either solution triangle. We saw above these correspond to $X_1$ and $X_7$ of the reference, i.e., it is the Soddy line of the reference \cite{etc-central-lines}, which is known to pass through $X_{20}$. The Brocard axis shared by the A-excircle solutions are the A-exversion ($a \to -a$) of the incircle Brocard axis, and similarly for the B- and C-excircles. It can be shown these 4 lines meet at $X_{20}$.
\end{proof}

\subsection*{Twenty-three from one}

There are a total of 4 pairs of triangles which are solutions to the CCP on both incircle and excircle, i.e., there are eight triangles and a total of 24 vertices.

\begin{proposition}
Twenty-three of said vertices can be directly derived from a single vertex of a solution triangle of the CCP in the incircle. 
\end{proposition}

\begin{proof}
As seen in \cref{prop:two-sols}, the A-vertex of the first solution $T_1$ on the incircle is given by:
\[\left[(1-\phi)^2 (s-b)(s-c), (s-a) (s-c),(2-\phi)^2 (s-a)(s-b)\right]\]

Perform a bicentric substitution, i.e., $b \to c$ and $c\to b$, and swap postions 2 and 3 to arrive at:
\[\left[(\phi-1)^2 (s-b)(s-c),(\phi-2)^2 (s-a)(s-c),(s-a)(s-b)\right]\]
i.e., the A-vertex of $T_2$. The other vertices can be computed cyclically. Now, derive the A-excircle solution from the incircle one by performing an ``A-exversion'', i.e., changing every $a \to -a$, obtaining the $T_1$ A-vertex of the A-excircle:
\[\left[ (\phi-1)^2 (b-s)(s-c),(s-b) s, (\phi-2)^2 (s-c) s\right]\]

For the B-vertex of the A-excircle, we first perform a cyclic substitution, then the exversion $a \to -a$, obtaining:
\[\left[(\phi-1)^2 (s-b) (s-c) , (s-a)(s-c), (\phi-2)^2 (s-a)(s-b) \right]\]

which leads to:
\[\left[(\phi-2)^2 (s-b)(s-c), (\phi-1)^2 (s-a)(s-c),(s-a)(s-b)\right]\]

and then: 
\[\left[(\phi-2)^2 (s-b)(s-c) , (\phi-1)^2 s(b-s),s(c-s)\right]\]

The remaining vertices can be obtained similarly.
\end{proof}

%% file: 210_ack.tex
\noindent We would like to thank Arseniy Akopyan and Ronaldo Garcia for useful discussions. 

%% file: 230_cc_correspond.tex
Let $[i,k]$ indicate that $X_i$ of either solution of the incircle CCP coincides with $X_j$ of the reference triangle. The following is a list of corresponding pairs: {\small $[3,1]$, $[6,7]$, $[15,3638]$,  $[16,3639]$,  $[32,10481]$,  $[182,5542]$,  $[187,1323]$,  $[371,482]$,  $[372,481]$,  $[511,516]$,  $[512,514]$,  $[575,43180]$,  $[576,30424]$,  $[1151,176]$,  $[1152,175]$,  $[1350,390]$,  $[1351,4312]$,  $[1384,21314]$,  $[2076,14189]$,  $[3053,279]$,  $[3098,30331]$,  $[3311,1373]$,  $[3312,1374]$,  $[3592,21169]$,  $[5017,42309]$,  $[5023,3160]$,  $[5085,11038]$,  $[5585,31721]$,  $[5611,10652]$,  $[5615,10651]$,  $[6200,31538]$,  $[6221,1371]$,  $[6396,31539]$,  $[6398,1372]$,  $[6409,17805]$,  $[6410,17802]$,  $[6419,21171]$,  $[6425,31601]$,  $[6426,31602]$,  $[6431,21170]$,  $[6437,17804]$,  $[6438,17801]$,  $[6449,17806]$,  $[6450,17803]$,  $[11824,31567]$,  $[11825,31568]$,  $[12305,30333]$,  $[12306,30334]$,  $[14810,43179]$,  $[15815,5543]$,  $[21309,20121]$,  $[31884,8236]$,  $[43118,30342]$,  $[43119,30341]$,  $[43120,31570]$,  $[43121,31569]$}.